\documentclass[a4paper,11pt]{amsart}
\pdfoutput=1 
\usepackage{amsmath,amsthm,amssymb}
\usepackage[mathscr]{eucal}
 \usepackage{cite}
\usepackage{upgreek}
\usepackage[bookmarks=false]{hyperref}
\usepackage{enumerate}
\usepackage{amsmath}
\usepackage{mathrsfs}
\usepackage{blindtext}
\usepackage{scrextend}
\usepackage{enumitem}
\usepackage{bm} 
\addtokomafont{labelinglabel}{\sffamily}
\usepackage{color}
\usepackage[at]{easylist}
\usepackage{tikz}
\usepackage{svg}
\usepackage{graphicx}

\usepackage{comment}

\numberwithin{equation}{section}

\newtheorem{main}{Theorem}

\newtheorem{thm}{Theorem}[section]
\newtheorem*{thm*}{Theorem}
\newtheorem{lem}[thm]{Lemma}
\newtheorem*{prob*}{Problem}

\newtheorem{prop}[thm]{Proposition}
\newtheorem*{prop*}{Proposition}

\newtheorem{cor}[thm]{Corollary}
\newtheorem*{cor*}{Corollary}

\theoremstyle{definition}
\newtheorem{defn}[thm]{Definition}
\newtheorem*{defn*}{Definition}

\newtheorem{rem}[thm]{Remark}

\newtheorem{question}[thm]{Question}

\newtheorem*{question*}{Question}
\newtheorem*{Pquestion*}{Popa's question}

\newtheorem*{conv*}{Convention}

\newcommand{\Sym}{\textrm{Sym}}
\begin{document}

\title[ Soficity for group actions on sets and applications]
{ Soficity for group actions on sets and applications}

\author[David Gao]{David Gao}
\address{Department of Mathematical Sciences, UCSD, 9500 Gilman Dr, La Jolla, CA 92092, USA}\email{weg002@ucsd.edu}
\author[Srivatsav Kunnawalkam Elayavalli]{Srivatsav Kunnawalkam Elayavalli}
\address{Department of Mathematical Sciences, UCSD, 9500 Gilman Dr, La Jolla, CA 92092, USA}\email{srivatsav.kunnawalkam.elayavalli@vanderbilt.edu}
\urladdr{https://sites.google.com/view/srivatsavke/home}

\author[Gregory Patchell]{Gregory Patchell}
\address{Department of Mathematical Sciences, UCSD, 9500 Gilman Dr, La Jolla, CA 92092, USA}\email{gpatchel@ucsd.edu}

\begin{abstract}
In this article we  develop a notion of soficity for actions of countable groups on sets. We show two  equivalent perspectives, several natural properties and examples. Notable examples include arbitrary actions of both amenable groups and free groups, and actions of sofic groups with locally finite stabilizers.   As applications we prove soficity for generalized wreath products (and amalgamated free generalized wreath products) of sofic groups where the underlying group action is sofic. This generalizes the result of Hayes and Sale \cite{HayesSale}, and proves soficity for many new families of groups.

\end{abstract}
\maketitle

\section{Introduction}

The study of finitary approximations of countable groups is a modern area of interest in group theory. One such finitary approximation property is the notion of soficity due to Gromov (for a detailed survey with references see \cite{surveysofic}). Groups that satisfy this property are known to additionally verify other important group theoretic questions such as Gottschalk surjunctivity conjecture \cite{KLi}, Kervaire conjecture (see Corollary 10.4 of \cite{surveysofic}), Kaplansky direct finiteness conjecture \cite{elekszabodirect}, determinant conjecture \cite{Luck}, Connes embedding problem \cite{Connes} etc (see also \cite{Bowenfinvariant, ThomDiophantine}). There is therefore an interest to identify more examples of soficity in countable groups, irrespective of the notorious open problem of the existence of a non sofic group.  There are currently many known examples of sofic groups (two elementary and important sources of examples include amenable groups, residually finite groups), and examples of group operations that preserve soficity. These include taking direct products of sofic groups, amalgamated free products and HNN extensions of sofic groups over amenable amalgams \cite{Paunescusofic1, elekszabosofic, DKPsofic, Popaindep}, wreath products of sofic groups \cite{HayesSale}. 

In this paper we strictly generalize \cite{HayesSale} and identify new examples of soficity in the setting of generalized wreath product groups. In order to identify the correct level of generality, we first had to understand and develop a natural notion of soficity for a group action on a set. To our knowledge, a satisfactory definition of this is currently unavailable in the literature (note that there are very satisfying definitions of soficity in the different more analytic setting of probability measure preserving actions, see \cite{ElekLip, PaunConvex}). Aside from our applications in this paper, we would like to place an emphasis on the fact that we fill this void, and begin developing a fruitful theory of soficity for group actions on sets, with potential other use in the future. We provide the definition below: 

Let $G$ be a countable discrete group, $X$ be a countable discrete set, $\alpha: G \curvearrowright X$ be an action, $A$ be a finite set, $\varphi: G \rightarrow \Sym(A)$ be a map. For a finite subset $F \subseteq G$ and $\epsilon > 0$, $\varphi$ is called $(F, \epsilon)$-\textit{multiplicative} if $d(\varphi(gh), \varphi(g)\varphi(h)) < \epsilon$ for all $g, h \in F$, where $d$ denotes the normalized Hamming distance. For finite subsets $F \subseteq G$, $E \subseteq X$, and $\epsilon > 0$, $\varphi$ is called an $(F, E, \epsilon)$-\textit{orbit approximation of} $\alpha$ if there exists a finite set $B$ and a subset $S \subseteq A$ s.t. $|S| > (1 - \epsilon)|A|$ and for each $s \in S$ there is an injective map $\pi_s: E \hookrightarrow B$ s.t. $\pi_{\varphi(g)s}(x) = \pi_s(\alpha(g^{-1})x)$ for all $s \in S$, $g \in F$, $x \in E$, whenever $\varphi(g)s \in S$ and $\alpha(g^{-1})x \in E$. 

\begin{defn*}[see Definition \ref{main defn}]
    
$\alpha$ is called \textit{sofic} if for all finite subsets $F \subseteq G$, $E \subseteq X$, and $\epsilon > 0$, there exists a finite set $A$ and a map $\varphi: G \rightarrow \Sym(A)$ which is unital, $(F, \epsilon)$-multiplicative, and an $(F, E, \epsilon)$-orbit approximation of $\alpha$.

\end{defn*}

If the reader is more comfortable working in the ultrapower framework, we redirect them to Proposition \ref{ultra-prod}, where we present a natural equivalent definition involving the natural action of the universal sofic group on the Loeb measure space (see for instance \cite{hayes2023sofic}). This could even serve as an additional motivation for why we define soficity in this manner. 

The first example  is the left action of a sofic group $G$ on itself. We show (in Theorem \ref{thm-left-multi} and Theorem \ref{group being sofic if action is sofic}) that this is sofic if and only if $G$ is sofic. More generally, we are able to show (in Theorem \ref{thm-left-multi}) that $G$ acting on the coset space $G/H$ is sofic where $H$ is any locally finite subgroup. On the other side, we have (see Theorem \ref{group being sofic if action is sofic}) that if the action of $G$ on a coset space $G/H$ is sofic, then there is a normal subgroup $N$ of $G$ such that $N\leq H$ and $G/N$ is sofic.  This by itself could give help to finding new examples of sofic groups. \footnote{We would like to remind the readers that it remains a puzzling question whether $G/N$ is sofic for $G$ sofic and $N$ is a finite normal subgroup, say even $|N|=2$.}

The following are rather quick observations:  If $\alpha: G \curvearrowright X$ is the composition of a quotient map $q: G \rightarrow H$ and a sofic action $\beta: H \curvearrowright X$, then $\alpha$ is sofic;
     If $\alpha: G \curvearrowright X$ is sofic, then the restriction of $\alpha$ to each of its orbits is sofic; If $\alpha: G \curvearrowright X$ is sofic and $H$ is a subgroup of $G$, then $\alpha|_H$ is sofic; If $G_1 \subseteq G_2 \subseteq \cdots$ is an increasing sequence of subgroups of $G$ whose union is $G$, and $\alpha: G \curvearrowright X$ restricted to each $G_i$ is sofic, then $\alpha$ is sofic. Moreover we have that if the restriction of $\alpha: G \curvearrowright X$ to each of its orbits is sofic, then $\alpha$ is sofic (see Proposition \ref{transitive-prop}). This  naturally reduces the study of sofic actions to transitive ones.

It is  of course an open question whether all actions of sofic groups  are sofic. Interestingly in Theorem \ref{amenable examples all}  and Theorem \ref{free examples all} we find that there are two classes of groups whose arbitrary actions are sofic: amenable groups and free groups.  We are ready to now state our main result (see Theorem \ref{gen-wr-prod}) which recovers and generalizes the main result of Hayes and Sale \cite{HayesSale}. 

\begin{main}
    Let $G, H$ be sofic groups, $\alpha: H \curvearrowright X$ be a sofic action. Then the generalized wreath product $G \wr_{\alpha} H$ is sofic.
\end{main}
    
The proof of the main result is very much inspired by and pushes the ideas of \cite{HayesSale}. The above result also applies to the setting of amalgamated free generalized wreath product (see Theorem \ref{gen-free-wr-prod}). The unfamiliar reader is directed to Definition \ref{defn fgwp}. At this point we would like to mention that the natural variant of the above result can also be proved in the context of hyperlinearity, see Theorems \ref{hyper1}, \ref{hyper2} and Corollary \ref{hyper3}. 

The following are  some remarks and questions we wish to highlight out before concluding the introduction. Let $(\Omega, \mu)$ be a standard probability space, $G$ be a sofic group, $\alpha: G \curvearrowright X$ be a sofic action. Then the induced generalized Bernoulli shift $G \curvearrowright (\Omega^X, \mu^{\otimes |X|})$ is sofic in the sense of \cite{Paunescusofic1}. This places an aspect of our work in the context of Elek-Lippner's work \cite{eleklippner} defining soficity for equivalence relations and of Paunescu \cite{Paunescusofic1}. The most  important  open question that arises from our work is whether the converse of our main result holds: Let $G, H$ be nontrivial countable groups, $\alpha: H \curvearrowright X$ be an action. Then the generalized wreath product $G \wr_{\alpha} H$ is sofic if and only if $G$ and $H$ are sofic and the action $\alpha$ is sofic? As we highlight in Section 4  a positive answer to this question follows from a positive answer to the following stability question: Suppose we have actions $\alpha_i: G_i \curvearrowright X$ which commute with each other and where $i$ ranges over a countable index set. The actions naturally give rise to an action $\alpha: \oplus_i G_i \curvearrowright X$. Then, is $\alpha$ sofic if and only if all $\alpha_i$ are sofic?


\section{Sofic actions}

In the following, $d$, when denoting a metric on a symmetric group of a finite set, shall always be understood to be the normalized Hamming distance, unless specified otherwise. The following is our finitary definition  of sofic actions: 

\begin{defn}\label{main defn}
Let $G$ be a countable discrete group, $X$ be a countable discrete set, $\alpha: G \curvearrowright X$ be an action, $A$ be a finite set, $\varphi: G \rightarrow \Sym(A)$ be a map (not necessarily a homomorphism):
\begin{enumerate}
    \item $\varphi$ is called \textit{unital} if $\varphi(1_G) = 1$;
    \item For a finite subset $F \subseteq G$ and $\epsilon > 0$, $\varphi$ is called $(F, \epsilon)$-\textit{multiplicative} if $d(\varphi(gh), \varphi(g)\varphi(h)) < \epsilon$ for all $g, h \in F$;
    \item For finite subsets $F \subseteq G$, $E \subseteq X$, and $\epsilon > 0$, $\varphi$ is called an $(F, E, \epsilon)$-\textit{orbit approximation of} $\alpha$ if there exists a finite set $B$ and a subset $S \subseteq A$ s.t. $|S| > (1 - \epsilon)|A|$ and for each $s \in S$ there is an injective map $\pi_s: E \hookrightarrow B$ s.t. $\pi_{\varphi(g)s}(x) = \pi_s(\alpha(g^{-1})x)$ for all $s \in S$, $g \in F$, $x \in E$, whenever $\varphi(g)s \in S$ and $\alpha(g^{-1})x \in E$;
    \item Recall that $G$ is called \textit{sofic} if for all finite subsets $F\subseteq G$ and $\epsilon > 0$ there exists a finite set $A$ and a map $\varphi: G\to \Sym(A)$ which is unital, $(F,\epsilon)$-multiplicative, and $d(1, \varphi(g)) > 1 - \epsilon$ for all $g\in F\setminus\{e\}.$
    \item $\alpha$ is called \textit{sofic} if for all finite subsets $F \subseteq G$, $E \subseteq X$, and $\epsilon > 0$, there exists a finite set $A$ and a map $\varphi: G \rightarrow \Sym(A)$ which is unital, $(F, \epsilon)$-multiplicative, and an $(F, E, \epsilon)$-orbit approximation of $\alpha$.
\end{enumerate}
\end{defn}

Now we will place the above definition in the non-finitary setting of ultraproducts. We need some preliminary definitions and notations. 

\begin{defn}
Let $(X_n)$ be a sequence of sets, $\mathcal{U}$ be a free ultrafilter on $\mathbb{N}$. Then $\prod_\mathcal{U} X_n$, the \textit{algebraic ultraproduct} of $(X_n)$, is defined as $\prod_\mathcal{U} X_n = \prod X_n/\sim$ where $f \sim g$ iff $\{n: f(n) = g(n)\} \in \mathcal{U}$. We shall write $(x_n)_\mathcal{U}$ to mean the element of $\prod_\mathcal{U} X_n$ represented by $(x_n) \in \prod X_n$. If $(A_n \subseteq X_n)$ is a sequence of subsets, then we shall write $(A_n)_\mathcal{U}$ to mean,
\begin{equation*}
    (A_n)_\mathcal{U} = \{(x_n)_\mathcal{U} \in \prod_\mathcal{U} X_n: \{n: x_n \in A_n\} \in \mathcal{U}\}
\end{equation*}
\end{defn}

\begin{defn}
Let $(X_n)$ and $(Y_n)$ be two sequences of sets, $\mathcal{U}$ be a free ultrafilter on $\mathbb{N}$. A map $\varphi: \prod_\mathcal{U} X_n \rightarrow \prod_\mathcal{U} Y_n$ is called \textit{liftable} if there exists a sequence of maps $\varphi_n: X_n \rightarrow Y_n$ such that $\varphi((x_n)_\mathcal{U}) = (\varphi_n(x_n))_\mathcal{U}$ for all $(x_n)_\mathcal{U} \in \prod_\mathcal{U} X_n$.
\end{defn}

\begin{defn}[see for instance \cite{PaunConvex}]
Let $[n] = \{1, \cdots, n\}$, $\mu_n$ be the normalized counting measure on $[n]$, $\mathcal{U}$ be a free ultrafilter on $\mathbb{N}$. Then the \textit{Loeb measure space}, denoted by $\prod_\mathcal{U} ([n], \mu_n)$, is defined to have the underlying set $\prod_\mathcal{U} [n]$. For a sequence $(A_n \subseteq [n])$, we define $\mu_\mathcal{U}((A_n)_\mathcal{U}) = \lim_\mathcal{U} \mu_n(A_n)$. This can be extended to a measure on the $\sigma$-algebra generated by all such $(A_n)_\mathcal{U}$.
\end{defn}

\begin{defn}
Let $(X_n, d_n)$ be a sequence of metric spaces, $\mathcal{U}$ be a free ultrafilter on $\mathbb{N}$. Then $\prod_\mathcal{U} (X_n, d_n)$, the \textit{metric ultraproduct} of $(X_n, d_n)$, is defined to have the underlying set $\prod X_n/\sim$ where $f \sim g$ iff $\lim_\mathcal{U} d_n(f(n), g(n)) = 0$. We shall write $(x_n)_\mathcal{U}$ to mean the element of $\prod_\mathcal{U} X_n$ represented by $(x_n) \in \prod X_n$. Then the metric on this space is defined by $d_\mathcal{U}((x_n)_\mathcal{U}, (y_n)_\mathcal{U}) = \lim_\mathcal{U} d_n(x_n, y_n)$.
\end{defn}

\begin{defn}
Let $\mathcal{U}$ be a free ultrafilter on $\mathbb{N}$. Let $X$ consists of all liftable maps $\prod_\mathcal{U} ([n], \mu_n) \rightarrow \prod_\mathcal{U} \mathbb{N}$. We observe that the set $\{x: f(x) = g(x)\}$ for any $f, g \in X$ is always of the form $(A_n)_\mathcal{U}$ and therefore measurable. Let $\mathbb{X}^1_\mathcal{U}$ be $X/\sim$ where $f \sim g$ iff they coincide a.e. Let $d_\mathcal{U}$ be the metric on $\mathbb{X}^1_\mathcal{U}$ given by $d_\mathcal{U}([f], [g]) = \mu_\mathcal{U}\{x: f(x) \neq g(x)\}$. We observe that the universal sofic group $\prod_\mathcal{U} (S_n, d)$ naturally acts on the Loeb measure space via pmp automorphisms, and that the pre-composition of a liftable map with a sequence of permutations results in a liftable map. Therefore, $\prod_\mathcal{U} (S_n, d)$ naturally acts on $\mathbb{X}^1_\mathcal{U}$ via pre-composition of inverses. This action shall be denoted by $\mathbb{S}^1_\mathcal{U}$ and called the \textit{first universal sofic action}. We observe that this action preserves the metric $d_\mathcal{U}$.
\end{defn}

\begin{defn}
Let $\mathcal{U}$ be a free ultrafilter on $\mathbb{N}$. For each $n$, let $X_n$ be the collection of all maps $[n] \rightarrow \mathbb{N}$ and $d_n$ be the normalized Hamming distance on $X_n$. Let $\mathbb{X}^2_\mathcal{U} = \prod_\mathcal{U} (X_n, d_n)$. We define the following action $\mathbb{S}^2_\mathcal{U}$ of the universal sofic group $\prod_\mathcal{U} (S_n, d)$ on $\mathbb{X}^2_\mathcal{U}$ by $\mathbb{S}^2_\mathcal{U}((g_n)_\mathcal{U})((f_n)_\mathcal{U}) = (f_n \circ g_n^{-1})_\mathcal{U}$. We easily observe that this is a well-defined action that preserves the metric $d_\mathcal{U}$. We shall call $\mathbb{S}^2_\mathcal{U}$ the \textit{second universal sofic action}.
\end{defn}

\begin{lem}
There is an isometric bijection $\iota: \mathbb{X}^2_\mathcal{U} \rightarrow \mathbb{X}^1_\mathcal{U}$ which is equivariant under the two universal sofic actions.
\end{lem}

\begin{proof} For each $(f_n)_\mathcal{U} \in \mathbb{X}^2_\mathcal{U}$, defined $\iota((f_n)_\mathcal{U})$ to be the map $\prod_\mathcal{U} ([n], \mu_n) \rightarrow \prod_\mathcal{U} \mathbb{N}$ defined by $\iota((f_n)_\mathcal{U})((x_n)_\mathcal{U}) = (f_n(x_n))_\mathcal{U}$. It is easy to verify that this indeed satisfies the conditions of the lemma.
\end{proof}

In light of the lemma above, we shall simply identify the two universal sofic actions and denote this action by $\mathbb{S}_\mathcal{U}: \prod_\mathcal{U} (S_n, d) \curvearrowright \mathbb{X}_\mathcal{U}$. Now  we present the natural ultraproduct definition of sofic actions. 

\begin{prop}\label{ultra-prod}
Let $G$ be a countable discrete group, $X$ be a countable discrete set, $\alpha: G \curvearrowright X$ be an action. The following are equivalent,
\begin{enumerate}
    \item $\alpha$ is sofic;
    \item There exists a free ultrafilter $\mathcal{U}$ on $\mathbb{N}$, a group homomorphism $\varphi: G \rightarrow \prod_\mathcal{U} (S_n, d)$, and a map $\pi: X \rightarrow \mathbb{X}_\mathcal{U}$ s.t. $\mathbb{S}_\mathcal{U}(\varphi(g))(\pi(x)) = \pi(\alpha(g)x)$ for all $x \in X$, $g \in G$ and s.t. $d_\mathcal{U}(\pi(x), \pi(y)) = 1$ for all $x \neq y \in X$.
\end{enumerate}
\end{prop}

\begin{proof} $(\Rightarrow)$ Fix increasing sequences of finite subsets $F_1 \subseteq F_2 \subseteq \cdots \subseteq G$ and $E_1 \subseteq E_2 \subseteq \cdots \subseteq X$ s.t. $\cup_i F_i = G$, $\cup_i E_i = X$. Fix a decreasing sequence $\epsilon_i > 0$ s.t. $\lim_i \epsilon_i = 0$. For each $i$, let $\varphi_i: G \rightarrow \Sym(A_i)$ be a unital, $(F_i, \epsilon_i)$-multiplicative, and an $(F_i, E_i, \epsilon_i)$-orbit approximation of $\alpha$. By taking the Cartesian products of $A_i$ with auxiliary finite sets if necessary, we may assume $|A_i|$ is strictly increasing. By definition of $(F_i, E_i, \epsilon_i)$-orbit approximation, there exists a finite set $B_i$ and a subset $S_i \subseteq A_i$ s.t. $|S_i| > (1 - \epsilon_i)|A_i|$ and for each $s \in S_i$ there is an injective map $\pi^i_s: E_i \hookrightarrow B_i$ s.t. $\pi^i_{\varphi_i(g)s}(x) = \pi^i_s(\alpha(g^{-1})x)$ for all $s \in S_i$, $g \in F_i$, $x \in E_i$, whenever $\varphi_i(g)s \in S_i$ and $\alpha(g^{-1})x \in E_i$. By embedding $B_i$ into $\mathbb{N}$ we may take the co-domain of $\pi^i_s$ to be $\mathbb{N}$. Let $\mathcal{U}$ be any free ultrafilter on $\mathbb{N}$ containing the set $\{|A_i|: i \geq 1\}$. We may then define $\varphi: G \rightarrow \prod_\mathcal{U} (S_n, d)$ by defining $\varphi(g) = (g_n)_\mathcal{U}$ with $g_n = \varphi_i(g)$ whenever $n = |A_i|$ and $g_n = 1$ otherwise. Since $\{|A_i|: i \geq 1\} \in \mathcal{U}$ and $\varphi_i$ is $(F_i, \epsilon_i)$-multiplicative, we see that $\varphi$ is a group homomorphism.

We then define $\pi: X \rightarrow \mathbb{X}_\mathcal{U}$ as follows: For each $x \in X$, $n \in \mathbb{N}$, define $\pi_{x, n}: [n] \rightarrow \mathbb{N}$,
\begin{equation*}
    \pi_{x, n}(s) = \begin{cases}
        \pi^i_s(x), \textrm{if }n = |A_i|\textrm{ and }x \in E_i\textrm{ and }s \in S_i\\
        0, \textrm{otherwise}
    \end{cases}
\end{equation*}
For each $x \in X$, $\pi(x)$ shall be represented by the sequence of maps $\pi_{x, n}$. We observe that as $\{|A_i|: i \geq 1\} \in \mathcal{U}$, it does not matter how $\pi_{x, n}$ is defined when $n \notin \{|A_i|: i \geq 1\}$. Now, let $x \neq y \in X$, then for large enough $i$ we have $x, y \in E_i$. For $s \in S_i$, we then have $\pi_{x, |A_i|}(s) = \pi^i_s(x) \neq \pi^i_s(y) = \pi_{y, |A_i|}(s)$ as $\pi^i_s$ is injective. As $\frac{|S_i|}{|A_i|} > 1 - \epsilon_i \rightarrow 1$, $d_{|A_i|}(\pi_{x, |A_i|}, \pi_{y, |A_i|}) \geq \frac{|S_i|}{|A_i|} \rightarrow 1$, so $d_\mathcal{U}(\pi(x), \pi(y)) = 1$.

Finally, fix $x \in X$, $g \in G$. Then for large $i$, $x \in E_i$, $g^{-1} \in F_i$, and $\alpha(g)x \in E_i$. Now, for any $s \in S_i \cap \varphi_i(g^{-1})^{-1}S_i$, by definition of $\pi_{x, n}$ and $(F_i, E_i, \epsilon_i)$-orbit approximation we have,
\begin{equation*}
    \pi_{x, |A_i|}(\varphi_i(g^{-1})s) = \pi^i_{\varphi_i(g^{-1})s}(x) = \pi^i_s(\alpha(g)x) = \pi_{\alpha(g)x, |A_i|}(s)
\end{equation*}

Since $\frac{|S_i \cap \varphi_i(g^{-1})^{-1}S_i|}{|A_i|} > 1 - 2\epsilon_i \rightarrow 1$, this means the maps given by $\mathbb{S}_\mathcal{U}(\varphi(g))(\pi(x))$ and by $\pi(\alpha(g)x)$ coincide on a set of measure 1 in $\prod_\mathcal{U} ([n], \mu_n)$, whence they are identified in $\mathbb{X}_\mathcal{U}$. This proves the claim.

$(\Leftarrow)$ For each $g \in G$, we shall write $\varphi(g) = (g_n)_\mathcal{U}$. Since $\varphi(1_G) = 1$, we shall in particular choose $\varphi(1_G) = (1)_\mathcal{U}$. Let $\varphi_n: G \rightarrow S_n$ be defined by $\varphi_n(g) = g_n$. Now, fix finite $F \subseteq G$, $E \subseteq X$, and $\epsilon > 0$, we shall show that there exists $n$ s.t. $\varphi_n$ is unital, $(F, \epsilon)$-multiplicative, and an $(F, E, \epsilon)$-orbit approximation of $\alpha$. We observe that $\varphi_n$ is unital for all $n$. Since $F$ is finite and $\varphi$ is a group homomorphism, there exists $L_1 \in \mathcal{U}$ s.t. for all $n \in L_1$, $\varphi_n$ is $(F, \epsilon)$-multiplicative.

Now, we observe that, in the definition of $(F, E, \epsilon)$-orbit approximation of $\alpha$, it is not necessary that $B$ is a finite set, as, for an infinite $B$, we may simply restrict $B$ to $\cup_{s \in S} \pi_s(E)$ and the latter set is finite. Thus, we may set $B = \mathbb{N}$. Choose $\epsilon' > 0$ s.t. $1 - |E|^2\epsilon' - |F||E|\epsilon' \geq 1 - \epsilon$. Now, for each $x \in E$, we represent $\pi(x)$ as a sequence of maps $(\pi_{x, n})_\mathcal{U}$ with $\pi_{x, n}: [n] \rightarrow \mathbb{N} = B$. We first observe that, for any $x \neq y \in E$, as $d_\mathcal{U}(\pi(x), \pi(y)) = 1$, there exists $L_{2, x, y} \in \mathcal{U}$ s.t. $d_n(\pi_{x, n}, \pi_{y, n}) > 1 - \epsilon'$ for all $n \in L_{2, x, y}$. Let $L_2 = \cap_{x \neq y \in E} L_{2, x, y}$. Since $E$ is finite, $L_2 \in \mathcal{U}$. For each $n \in L_2$, let,
\begin{equation*}
    \tilde{S}_n = \cap_{x \neq y \in E} \{s \in [n]: \pi_{x, n}(s) \neq \pi_{y, n}(s)\}
\end{equation*}

By assumption $\frac{|\tilde{S}_n|}{n} > 1 - |E|^2\epsilon'$. For each $s \in \tilde{S}_n$, defined $\pi^n_s: E \rightarrow B$ by $\pi^n_s(x) = \pi_{x, n}(s)$. Then $\pi^n_s$ is injective for all $s \in \tilde{S}_n$.

Now, fix any $g \in F$, $x \in E$ with $\alpha(g^{-1})x \in E$. Recall that $\mathbb{S}_\mathcal{U}(\varphi(g^{-1}))(\pi(x)) = \pi(\alpha(g^{-1})x)$. $\mathbb{S}_\mathcal{U}(\varphi(g^{-1}))(\pi(x))$ is represented by the sequence of maps $\pi_{x, n} \circ \varphi_n(g)$ while $\pi(\alpha(g^{-1})x)$ is represented by the sequence of maps $\pi_{\alpha(g^{-1})x, n}$. Thus, there exists $L_{3, g, x} \in \mathcal{U}$ s.t. $d_n(\pi_{x, n} \circ \varphi_n(g), \pi_{\alpha(g^{-1})x, n}) < \epsilon'$ for all $n \in L_{3, g, x}$. Let $L_3 = \cap_{g \in F, x \in E \cap \alpha(g)E} L_{3, g, x}$. Again, as $F$ and $E$ are finite, $L_3 \in \mathcal{U}$. For any $n \in L_2 \cap L_3$, define,
\begin{equation*}
    S_n = \tilde{S}_n \cap \bigcap_{g \in F, x \in E \cap \alpha(g)E} \{s \in [n]: [\pi_{x, n} \circ \varphi_n(g)](s) = \pi_{\alpha(g^{-1})x, n}(s)\}
\end{equation*}

Then $\frac{|S_n|}{n} > 1 - |E|^2\epsilon' - |F||E|\epsilon' \geq 1 - \epsilon$. By definition, $\pi^n_{\varphi_n(g)s}(x) = \pi^n_s(\alpha(g^{-1})x)$ for all $s \in S_n$, $g \in F$, $x \in E$, whenever $\varphi_n(g)s \in S_n$ and $\alpha(g^{-1})x \in E$. This shows that for all $n \in L_1 \cap L_2 \cap L_3 \neq \varnothing$, $\varphi_n$ is unital, $(F, \epsilon)$-multiplicative, and an $(F, E, \epsilon)$-orbit approximation of $\alpha$.
\end{proof}

\begin{rem}
If we regard the discrete set $X$ as equipped with the discrete metric, i.e., $d_X(x, y) = 1$ whenever $x \neq y$, then the distance condition in the second condition of the proposition above can be reduced to saying that $\pi$ is isometric. This inspires the following generalization of sofic actions to actions on separable metric spaces:
\end{rem}

\begin{defn}
Let $G$ be a countable discrete group, $X$ be a separable metric space with diameter bounded by 1, $\alpha: G \curvearrowright X$ be an isometric action. Then $\alpha$ is called \textit{sofic} if there exists a free ultrafilter $\mathcal{U}$ on $\mathbb{N}$, a group homomorphism $\varphi: G \rightarrow \prod_\mathcal{U} (S_n, d)$, and an isometric embedding $\pi: X \rightarrow \mathbb{X}_\mathcal{U}$ s.t. $\mathbb{S}_\mathcal{U}(\varphi(g))(\pi(x)) = \pi(\alpha(g)x)$ for all $x \in X$, $g \in G$.
\end{defn}

In the case of countable sets equipped with the discrete metric, this simply reduces to our previous definition of sofic actions.

\begin{thm}\label{group being sofic if action is sofic}
Let $G$ be a countable discrete group, $H \leq G$ be a subgroup. If the left multiplication action $\alpha: G \curvearrowright G/H$ is sofic, then there exists a normal subgroup $N \trianglelefteq G$ s.t. $N \leq H$ and $G/N$ is sofic. In particular, if the left multiplication action $\alpha: G \curvearrowright G$ is sofic, then $G$ is sofic.
\end{thm}

\begin{proof} By Proposition \ref{ultra-prod}, there exists a free ultrafilter $\mathcal{U}$ on $\mathbb{N}$, a group homomorphism $\varphi: G \rightarrow \prod_\mathcal{U} (S_n, d)$, and an injective map $\pi: G/H \rightarrow \mathbb{X}_\mathcal{U}$ s.t. $\mathbb{S}_\mathcal{U}(\varphi(g))(\pi(x)) = \pi(\alpha(g)x)$ for all $x \in G/H$, $g \in G$. Take $N = \textrm{ker}(\varphi)$. Then $G/N$ embeds into $\prod_\mathcal{U} (S_n, d)$ and is thus sofic. Assume $N \not\leq H$. We may then let $g \in N \setminus H$. $\varphi(g) = 1$, so $\pi(H) = \mathbb{S}_\mathcal{U}(\varphi(g))(\pi(H)) = \pi(\alpha(g)H) = \pi(gH)$. As $\pi$ is injective, $H = gH$, a contradiction. Thus, we must have $N \leq H$.
\end{proof}

The converse of the ``in particular" part of the above theorem is also true. In fact, we shall prove a stronger result in Theorem \ref{thm-left-multi}. First, though, we need a lemma:

\begin{lem}\label{lem-reln-bw-defns}
Suppose $G$ is a sofic group. Let $F \subseteq G$ be a finite subset and $\epsilon > 0$. Then there exists a finite set $A$ and a unital, $(F, \epsilon)$-multiplicative map $\varphi: G \rightarrow \Sym(A)$ such that $|S| > (1 - \epsilon)|A|$, where we define,
\begin{equation*}
\begin{split}
    S_1 &= \{s \in A: \varphi(g)s \neq \varphi(h)s, \forall g, h \in F, g \neq h\},\\
    S_2 &= \{s \in A: \varphi(gh)s = \varphi(g)\varphi(h)s, \forall g, h \in F\}, and\\
    S &= S_1 \cap S_2.
\end{split}
\end{equation*}
\end{lem}

\begin{proof} Assume WLOG that $F \subseteq G$ is a symmetric finite subset of $G$ containing the identity. Let $F' = F \cdot F = \{gh: g, h \in F\}$ and $\epsilon' = \frac{\epsilon}{4|F|^2}$.

Since $G$ is sofic, there exists a finite set $A$ and a map $\varphi: G \rightarrow \Sym(A)$ which is unital and $(F', \epsilon')$-multiplicative and satisfies $d(1, \varphi(g)) > 1 - \epsilon'$ for all non-identity $g \in F'$. 

Fix $g, h \in F$, $g \neq h$, then, since $F$ is symmetric and since $\varphi$ is unital and $(F', \epsilon')$-multiplicative, we have,
\begin{equation*}
\begin{split}
    d(\varphi(g)^{-1}, \varphi(g^{-1})) &= d(\varphi(g)\varphi(g)^{-1}, \varphi(g)\varphi(g^{-1}))\\
    &= d(1, \varphi(g)\varphi(g^{-1}))\\
    &= d(\varphi(gg^{-1}), \varphi(g)\varphi(g^{-1}))\\
    &< \epsilon'.
\end{split}
\end{equation*}

Thus, since $g^{-1}h \in F'$ and is not the identity,
\begin{equation*}
\begin{split}
    d(1, \varphi(g)^{-1}\varphi(h)) &\geq d(1, \varphi(g^{-1}h)) - d(\varphi(g^{-1}h), \varphi(g^{-1})\varphi(h)) - d(\varphi(g^{-1})\varphi(h), \varphi(g)^{-1}\varphi(h))\\
    &> 1 - \epsilon' - \epsilon' - d(\varphi(g^{-1}), \varphi(g)^{-1})\\
    &> 1 - 3\epsilon'.
\end{split}
\end{equation*}

Hence, $|\{s \in A: \varphi(g)s \neq \varphi(h)s\}| = |\{s \in A: s \neq \varphi(g)^{-1}\varphi(h)s\}| > (1 - 3\epsilon')|A|$. So,
\begin{equation*}
    |S_1| = |\bigcap_{\begin{smallmatrix}
    g, h \in F\\
    g \neq h
    \end{smallmatrix}} \{s \in A: \varphi(g)s \neq \varphi(h)s\}| > (1 - 3|F|^2\epsilon')|A|
\end{equation*}

We also observe that, for any fixed $g, h \in F$, $|\{s \in A: \varphi(gh)s = \varphi(g)\varphi(h)s\}| > (1 - \epsilon')|A|$. Therefore,
\begin{equation*}
    |S_2| = |\bigcap_{g, h \in F} \{s \in A: \varphi(gh)s = \varphi(g)\varphi(h)s\}| > (1 - |F|^2\epsilon')|A|
\end{equation*}

As such, $|S| = |S_1 \cap S_2| > (1 - 4|F|^2\epsilon')|A| = (1 - \epsilon)|A|$ as desired.
\end{proof}

\begin{thm}\label{thm-left-multi}
Let $G$ be a sofic group, $N \leq G$ be a locally finite subgroup of $G$. Then the left multiplication action $\alpha: G \curvearrowright G/N$ is sofic.
\end{thm}

\begin{proof} Let $F \subseteq G$, $E \subseteq G/N$ be finite subsets. Fix $\epsilon > 0$. Let $q: G \rightarrow G/N$ be the natural quotient map. Fix $\sigma: G/N \rightarrow G$ an arbitrary section of $q$. Let $U = \{\sigma(x)^{-1}g\sigma(g^{-1}x): x \in E, g \in F\}$. We observe that $U$ is finite and $U \subseteq N$. Let $N' = \langle U \rangle$. Since $N$ is locally finite, $N'$ is a finite group. Let,
\begin{equation*}
    F' = F \cup N' \cup (N' \cdot \sigma(E)^{-1})
\end{equation*}

We observe that since $N'$ is a subgroup, $1_G \in N'$, so $\sigma(E)^{-1} \subseteq N' \cdot \sigma(E)^{-1} \subseteq F'$. Let $\epsilon' = \frac{\epsilon}{|E| + 1}$. Since $F'$ is finite, by Lemma \ref{lem-reln-bw-defns}, there is a finite set $A$ and a unital, $(F', \epsilon')$-multiplicative map $\varphi: G \rightarrow \Sym(A)$ such that $|S'| > (1 - \epsilon')|A|$, where we define,
\begin{equation*}
\begin{split}
    S_1 &= \{s \in A: \varphi(g)s \neq \varphi(h)s, \forall g, h \in F', g \neq h\},\\
    S_2 &= \{s \in A: \varphi(gh)s = \varphi(g)\varphi(h)s, \forall g, h \in F'\}, and\\
    S' &= S_1 \cap S_2.
\end{split}
\end{equation*}

Now, on the set $S'$, we define a relation $s_1 \sim s_2$ if there exists $n \in N'$ s.t. $s_1 = \varphi(n)s_2$. Since $\varphi$ is unital, $\sim$ is reflexive. As $N'$ is a group, $N' \subseteq F'$, and by the definition of $S_2$, we see that $\sim$ is symmetric and transitive. Hence, $\sim$ is an equivalence relation. Let $B = S'/\sim$ and,
\begin{equation*}
    S = S' \cap \bigcap_{x \in E} \varphi(\sigma(x)^{-1})^{-1}S'
\end{equation*}

We observe that as $|S'| > (1 - \epsilon')|A|$ and $\epsilon' = \frac{\epsilon}{|E| + 1}$, we have $|S| > (1 - \epsilon)|A|$. Now, for $s \in S$, we define $\pi_s: E \hookrightarrow B$ to be $\pi_s(x) = [\varphi(\sigma(x)^{-1})s]_\sim$. We observe that this is injective. Indeed, assume $\pi_s(x) = \pi_s(y)$, i.e., $\varphi(\sigma(x)^{-1})s \sim \varphi(\sigma(y)^{-1})s$. Then there exists $n \in N'$ s.t. $\varphi(\sigma(x)^{-1})s = \varphi(n)\varphi(\sigma(y)^{-1})s$. Since $\sigma(E)^{-1} \subseteq F'$ and $N' \subseteq F'$, by the definition of $S_2$ we have $\varphi(\sigma(x)^{-1})s = \varphi(n)\varphi(\sigma(y)^{-1})s = \varphi(n\sigma(y)^{-1})s$. But then, as $n\sigma(y)^{-1} \in N' \cdot \sigma(E)^{-1} \subseteq F'$ and $\sigma(x)^{-1} \in \sigma(E)^{-1} \subseteq F'$, we have, by the definition of $S_1$, that $\sigma(x)^{-1} = n\sigma(y)^{-1}$. But then $\sigma(x) = \sigma(y)n^{-1} \in \sigma(y)N' \subseteq \sigma(y)N = y$, i.e., $x = y$. This shows that $\pi_s$ must be injective.

Finally, for all $s \in S$, $g \in F$, $x \in E$, if $\varphi(g)s \in S$ and $\alpha(g^{-1})x = g^{-1}x \in E$, then as $\sigma(E)^{-1} \subseteq F'$ and $F \subseteq F'$, and by the definition of $S_2$,
\begin{equation*}
    \pi_{\varphi(g)s}(x) = [\varphi(\sigma(x)^{-1})\varphi(g)s]_\sim = [\varphi(\sigma(x)^{-1}g)s]_\sim
\end{equation*}
and,
\begin{equation*}
    \pi_s(g^{-1}x) = [\varphi(\sigma(g^{-1}x)^{-1})s]_\sim
\end{equation*}

It now suffices to prove $\varphi(\sigma(x)^{-1}g)s \sim \varphi(\sigma(g^{-1}x)^{-1})s$. Let $n = \sigma(x)^{-1}g\sigma(g^{-1}x)$. By definition we have $n \in U \subseteq N'$. Since $g^{-1}x \in E$, as $\sigma(E)^{-1} \subseteq F'$ and $N' \subseteq F'$, we have, by definition of $S_2$,
\begin{equation*}
    \varphi(n)\varphi(\sigma(g^{-1}x)^{-1})s = \varphi(n\sigma(g^{-1}x)^{-1})s = \varphi(\sigma(x)^{-1}g)s
\end{equation*}

This concludes the proof.
\end{proof}

We record here some easy observations:

\begin{prop}\label{basic-prop}
\textrm{ }
\begin{enumerate}
    \item If $\alpha: G \curvearrowright X$ is the composition of a quotient map $q: G \rightarrow H$ and a sofic action $\beta: H \curvearrowright X$, then $\alpha$ is sofic;
    \item If $\alpha: G \curvearrowright X$ is sofic, then the restriction of $\alpha$ to each of its orbits is sofic;
    \item If $\alpha: G \curvearrowright X$ is sofic and $H$ is a subgroup of $G$, then $\alpha|_H$ is sofic;
    \item If $G_1 \subseteq G_2 \subseteq \cdots$ is an increasing sequence of subgroups of $G$ whose union is $G$, and $\alpha: G \curvearrowright X$ restricted to each $G_i$ is sofic, then $\alpha$ is sofic.
\end{enumerate}
\end{prop}

The converse of item 2 of the above proposition is also true, which naturally reduces the study of sofic actions to transitive ones:

\begin{prop}\label{transitive-prop}
If the restriction of $\alpha: G \curvearrowright X$ to each of its orbits is sofic, then $\alpha$ is sofic.
\end{prop}

\begin{proof} Fix finite subsets $F \subseteq G$, $E \subseteq X$, and $\epsilon > 0$. We need to show there exists a unital, $(F, \epsilon)$-multiplicative, $(F, E, \epsilon)$-orbit approximation of $\alpha$. Since $E$ is finite, it only intersects with finitely many orbits of $\alpha$, which we shall denote by $X_1, \cdots, X_n$. Let $E_i = E \cap X_i$. Choose $\epsilon' > 0$ s.t. $(1 - \epsilon')^n \geq 1 - \epsilon$. Since $\alpha|_{X_i}$ is sofic, there exists $\varphi_i: G \rightarrow \Sym(A_i)$ which is unital, $(F, \epsilon')$-multiplicative, and an $(F, E_i, \epsilon')$-orbit approximation of $\alpha|_{X_i}$. Define $\varphi: G \rightarrow \Sym(A_1 \times \cdots \times A_n)$ by,
\begin{equation*}
    \varphi(g)(a_1, \cdots, a_n) = (\varphi_1(g)a_1, \cdots \varphi_n(g)a_n)
\end{equation*}

It is clear that $\varphi$ is unital. We claim that $\varphi$ is $(F, \epsilon)$-multiplicative. Indeed, for $g, h \in F$,
\begin{equation*}
\begin{split}
    |A_1 \times \cdots \times A_n| \cdot (1 - d(\varphi(gh), \varphi(g)\varphi(h))) &= \prod_{i = 1}^n |\{a \in A_i: \varphi_i(gh)a = \varphi_i(g)\varphi_i(h)a\}|\\
    &= \prod_{i = 1}^n |A_i| \cdot (1 - d(\varphi_i(gh), \varphi_i(g)\varphi_i(h)))\\
    &> |A_1 \times \cdots \times A_n| \cdot (1 - \epsilon')^n\\
    &\geq |A_1 \times \cdots \times A_n| \cdot (1 - \epsilon)
\end{split}
\end{equation*}
so $d(\varphi(gh), \varphi(g)\varphi(h)) < \epsilon$.

We now show $\varphi$ is an $(F, E, \epsilon)$-orbit approximation of $\alpha$ to conclude the proof. We note that, as $\varphi_i$ is an $(F, E_i, \epsilon')$-orbit approximation of $\alpha|_{X_i}$, there exists a finite set $B_i$, $S_i \subseteq A_i$ and, for each $s \in S_i$, $\pi_s^i: E_i \hookrightarrow B_i$ satisfying the condition for an $(F, E_i, \epsilon')$-orbit approximation of $\alpha|_{X_i}$. Let $S = S_1 \times \cdots \times S_n$. Then $|S| > (1 - \epsilon')^n|A_1 \times \cdots \times A_n| \geq (1 - \epsilon)|A_1 \times \cdots \times A_n|$. Let $B = \coprod_{i=1}^n B_i$. For each $s = (s_1,  \cdots, s_n) \in S$, define $\pi_s: E = \coprod_{i=1}^n E_i \hookrightarrow B$ by,
\begin{equation*}
    \pi_s(x) = \pi_{s_i}^i(x), \textrm{when }x \in E_i
\end{equation*}

It is easy to see that the map is indeed injective. Finally, fix $s = (s_1, \cdots, s_n) \in S$, $g \in F$, $x \in E_i \subseteq E$. Assume $\varphi(g)s \in S$ and $\alpha(g^{-1})x \in E$. Then $\alpha|_{X_i}(g^{-1})x = \alpha(g^{-1})x \in E_i$. Also, $\varphi(g)s = (\varphi_1(g)s_1, \cdots, \varphi_n(g)s_n) \in S = S_1 \times \cdots \times S_n$ implies $\varphi_i(g)s_i \in S_i$, so,
\begin{equation*}
    \pi_{\varphi(g)s}(x) = \pi_{\varphi_i(g)s_i}^i(x) = \pi_{s_i}^i(\alpha|_{X_i}(g^{-1})x) = \pi_{s_i}^i(\alpha(g^{-1})x) = \pi_s(\alpha(g^{-1})x)
\end{equation*}
This concludes the proof.
\end{proof}

Combining Theorem \ref{thm-left-multi} and Proposition \ref{transitive-prop}, we see that all actions by sofic groups with locally finite stabilizers are sofic. We are unable to settle the more general case of amenable stabilizers. 

It is still open whether all actions by sofic groups are sofic. However, this does hold for amenable groups and free groups.

\begin{thm}\label{amenable examples all}
Any action $\alpha: G \curvearrowright X$ where $G$ is an amenable group is sofic.
\end{thm}

\begin{proof} Fix finite subsets $F \subseteq G$, $E \subseteq X$, and $\epsilon > 0$. Assume WLOG that $F$ is symmetric and contains the identity. Let $F' = F \cdot F$, $\epsilon' = \frac{\epsilon}{2|F|}$. As $G$ is amenable, we may then choose a Følner set $A \subseteq G$ with $|A \vartriangle gA| < \epsilon'|A|$ for all $g \in F'$. For $g \in G$, we choose $\varphi(g) \in \Sym(A)$ to be any element of $\Sym(A)$ s.t. $\varphi(g)a = ga$ whenever $ga \in A$. It is clear that $\varphi: G \rightarrow \Sym(A)$ is unital. We claim it is $(F, \epsilon)$-multiplicative. Indeed, let $g, h \in F$. For all $a \in A$ with $ha, gha \in A$, by definition we have $\varphi(g)\varphi(h)a = \varphi(gh)a$. So, as $h^{-1}, h^{-1}g^{-1} \in F \cdot F$,
\begin{equation*}
\begin{split}
    |A| \cdot d(\varphi(g)\varphi(h), \varphi(gh)) &\leq |\{a \in A: ha \notin A\}| + |\{a \in A: gha \notin A\}|\\
    &\leq |A \vartriangle h^{-1}A| + |A \vartriangle h^{-1}g^{-1}A|\\
    &< 2\epsilon'|A|\\
    &\leq \epsilon|A|
\end{split}
\end{equation*}

That is, $d(\varphi(g)\varphi(h), \varphi(gh)) < \epsilon$.

We now show $\varphi$ is an $(F, E, \epsilon)$-orbit approximation of $\alpha$ to conclude the proof. To this end, let $S = \{s \in A: gs \in A, \forall g \in F\} = \cap_{g \in F} (A \cap gA)$, as $F$ is symmetric. Since $|A \cap gA| > (1 - \epsilon')|A|$, we have $|S| > (1 - |F|\epsilon')|A| > (1 - \epsilon)|A|$. Now, let $B = \alpha(A^{-1}) \cdot E$. Define $\pi_s: E \hookrightarrow B$ by $\pi_s(x) = \alpha(s^{-1})x$. It is clear these maps are injective. We also have, for all $s \in S$, $g \in F$, $x \in E$ s.t. $\varphi(g)s \in S$ and $\alpha(g^{-1})x \in E$,
\begin{equation*}
    \pi_{\varphi(g)s}(x) = \pi_{gs}(x) = \alpha(s^{-1}g^{-1})x = \alpha(s^{-1})\alpha(g^{-1})x = \pi_s(\alpha(g^{-1})x)
\end{equation*}
where we have used the fact that, as $gs \in A$, by the definition of $\varphi(g)$, $\varphi(g)s = gs$. This concludes the proof.
\end{proof}

\begin{rem}
Together with item 1 of Proposition \ref{basic-prop}, this implies the action of any group on a finite set is sofic.
\end{rem}

\begin{thm}\label{free examples all}
Any action $\alpha: G \curvearrowright X$ where $G$ is a free group is sofic.
\end{thm}

\begin{proof} Fix finite subsets $F \subseteq G$, $E \subseteq X$, and $\epsilon > 0$. Let the free generators of $G$ be $\{g_1, g_2, \cdots\}$ (either a finite set or a sequence, depending on whether $G$ is finitely generated or not). Let $F^{-1} = \{f_1, \cdots, f_n\}$. We write $f_k = g_{i_{k, 1}}^{\epsilon_{k, 1}} \cdots g_{i_{k, m_k}}^{\epsilon_{k, m_k}}$, where $\epsilon$’s are in $\{\pm 1\}$. Let $B \subseteq X$ be a finite set containing $E$ and all elements of $X$ of the form $\alpha(g_{i_{k, l}}^{\epsilon_{k, l}} \cdots g_{i_{k, m_k}}^{\epsilon_{k, m_k}})x$ for all $x \in E$, $1 \leq k \leq n$, $1 \leq l \leq m_k$. We define a homomorphism $\psi: G \rightarrow \Sym(B)$ by defining, for a free generator $g_i$, $\psi(g_i)$ to be any element of $\Sym(B)$ s.t. $\psi(g_i)b = \alpha(g_i)b$ whenever $\alpha(g_i)b \in B$, then extending it to a homomorphism from $G$. Let $A = \Sym(B)$ regarded as a finite set. Then, the left multiplication action of $\Sym(B)$ on itself gives rise to an inclusion $\Sym(B) \hookrightarrow \Sym(A)$. Let $\varphi: G \rightarrow \Sym(A)$ be $\psi$ composed with this inclusion.

Since $\varphi$ is a homomorphism, it is unital and $(F, \epsilon)$-multiplicative. So it suffices to show it is an $(F, E, \epsilon)$-orbit approximation of $\alpha$. We let $S = A$ and for each $s \in S = \Sym(B)$, we define $\pi_s(x) = s^{-1}x$. It is clear this is an injective map from $E$ to $B$. Now, for $s \in S$, $g \in F$, $x \in E$, we have $\pi_{\varphi(g)s}(x) = \pi_{\psi(g)s}(x) = s^{-1}\psi(g^{-1})x$. Here, we note that by the definition of $\psi$ and $B$, it is easy to see that $\psi(g^{-1})x = \alpha(g^{-1})x$. Thus, whenever $\alpha(g^{-1})x \in E$, we have $\pi_{\varphi(g)s}(x) = s^{-1}\psi(g^{-1})x = s^{-1}\alpha(g^{-1})x = \pi_s(\alpha(g^{-1})x)$. This concludes the proof.
\end{proof}

\section{Generalized wreath products}

\begin{defn}
    Let $G,H$ be groups and $\alpha:H\curvearrowright X$ an action on a set. The \emph{generalized wreath product} $G\wr_\alpha H$ is the semidirect product $G^{\oplus X} \rtimes_\beta H$ where $\beta(h)((g_x)_{x \in X}) = (g_{\alpha(h)^{-1}(x)})_{x\in X}$. We observe that the same construction can also applied to a tracial von Neumann algebra $M$ in place of $G$ and with direct sums replaced by tensor products.
\end{defn}

\begin{defn}
    Let $G, H$ be groups, $\alpha: H \curvearrowright X$ be an action on a set, and $A \leq G$ be a subgroup. Then $H$ acts on the amalgamated free product $\ast^{x \in X}_A G_x$ where $G_x$ are copies of $G$ by permuting $G_x$ according to $\alpha$. Denote this action by $\beta$. Then the \textit{amalgamated free generalized wreath product}, denoted by $G \wr^{\ast_A}_\alpha H$ is given by $G \wr^{\ast_A}_\alpha H = (\ast^{x \in X}_A G_x) \rtimes_\beta H$. In case $A = \{1_G\}$, we shall call this group the \emph{free generalized wreath product} and denote it by $G \wr^\ast_\alpha H$. Note that the same construction can also be applied to an inclusion of tracial von Neumann algebras $N \subseteq M$ in place of the inclusion of groups $A \leq G$.
\end{defn}

The main applications in our present work in defining sofic actions are Theorem \ref{gen-wr-prod} and Theorem \ref{gen-free-wr-prod}, which are heavily inspired by the work of \cite{HayesSale}. First though, we need a few definitions and lemmas.

\begin{defn}\label{defn fgwp}
Let $G$ be a countable discrete group, $A$ be a finite set. Let $\alpha_A: \Sym(A) \curvearrowright A$ be the natural action. We may then consider the generalized wreath product $G \wr_{\alpha_A} \Sym(A)$, which is of the form $G^{\oplus A} \rtimes_{\beta_A} \Sym(A)$ where $\beta_A$ is the action $\beta_A: \Sym(A) \curvearrowright G^{\oplus A}$ induced by $\alpha_A$. We define a metric $d_{G, A}$ on this group given by
\begin{equation*}
    d_{G, A}(g_1\sigma_1, g_2\sigma_2) = \frac{1}{|A|} |\{a \in A: \sigma_1(a) \neq \sigma_2(a)\textrm{ or }g_1(\sigma_1(a)) \neq g_2(\sigma_2(a))\}|
\end{equation*}
where $g_i \in G^{\oplus A}$ and $\sigma_i \in \Sym(A)$, with the former regarded as functions from $A$ to $G$. It is easy to verify that this is indeed a metric and is invariant under both left and right multiplication.
\end{defn}

\begin{lem}
Let $G$ be a sofic group, $A$ be a finite set. Then there exists a free ultrafilter $\mathcal{U}$ on $\mathbb{N}$ and a sequence of finite sets $F_i$, s.t. $(G^{\oplus A} \rtimes_{\beta_A} \Sym(A), d_{G, A})$ embeds into $\prod_\mathcal{U} (\Sym(F_i), d)$ isometrically.
\end{lem}

\begin{proof} As $G$ is sofic, there exists a free ultrafilter $\mathcal{U}$ on $\mathbb{N}$, a sequence of finite sets $E_i$, and a group homomorphism $\phi: G \rightarrow \prod_\mathcal{U} (\Sym(E_i), d)$ s.t. $d_\mathcal{U}(\phi(g), \phi(h)) = 1$ whenever $g \neq h$. We lift it to a sequence of maps $\phi_i: G \rightarrow \Sym(E_i)$. Now, let $F_i = E_i \times A$. We define $\pi_i: G^{\oplus A} \rtimes_{\beta_A} \Sym(A) \rightarrow \Sym(F_i)$ by,
\begin{equation*}
    \pi_i(g\sigma)(x, a) = (\phi_i(g[\sigma(a)])x, \sigma(a))
\end{equation*}
where $g \in G^{\oplus A}$, $\sigma \in \Sym(A)$, $x \in E_i$, $a \in A$. Let $\pi: G^{\oplus A} \rtimes_{\beta_A} \Sym(A) \rightarrow \prod_\mathcal{U} (\Sym(F_i), d)$ be given by $\pi(g\sigma) = (\pi_i(g\sigma))_\mathcal{U}$. We first verify that $\pi$ is a group homomorphism. Indeed, let $g_1, g_2 \in G^{\oplus A}$, $\sigma_1, \sigma_2 \in \Sym(A)$, then,
\begin{equation*}
\begin{split}
    \pi_i(g_1\sigma_1g_2\sigma_2)(x, a) &= \pi_i(g_1\beta_{\sigma_1}(g_2)\sigma_1\sigma_2)(x, a)\\
    &= (\phi_i(g_1[\sigma_1\sigma_2(a)]\beta_{\sigma_1}(g_2)[\sigma_1\sigma_2(a)])x, \sigma_1\sigma_2(a))\\
    &= (\phi_i(g_1[\sigma_1\sigma_2(a)]g_2[\sigma_2(a)])x, \sigma_1\sigma_2(a))
\end{split}
\end{equation*}
while,
\begin{equation*}
\begin{split}
    \pi_i(g_1\sigma_1)\pi_i(g_2\sigma_2)(x, a) &= \pi_i(g_1\sigma_1)(\phi_i(g_2[\sigma_2(a)])x, \sigma_2(a))\\
    &= (\phi_i(g_1[\sigma_1\sigma_2(a)])\phi_i(g_2[\sigma_2(a)])x, \sigma_1\sigma_2(a))
\end{split}
\end{equation*}
so we have that,
\begin{equation*}
    d(\pi_i(g_1\sigma_1g_2\sigma_2), \pi_i(g_1\sigma_1)\pi_i(g_2\sigma_2)) = \frac{1}{|A|} \sum_{a \in A} d(\phi_i(g_1[\sigma_1\sigma_2(a)]g_2[\sigma_2(a)]), \phi_i(g_1[\sigma_1\sigma_2(a)])\phi_i(g_2[\sigma_2(a)]))
\end{equation*}

For any fixed $a \in A$, $g_1[\sigma_1\sigma_2(a)]$ is a fixed element of $G$, and so is $g_2[\sigma_2(a)]$, whence,
\begin{equation*}
    d(\phi_i(g_1[\sigma_1\sigma_2(a)]g_2[\sigma_2(a)]), \phi_i(g_1[\sigma_1\sigma_2(a)])\phi_i(g_2[\sigma_2(a)])) \rightarrow 0
\end{equation*}
as $i \rightarrow \mathcal{U}$. Hence, $d(\pi_i(g_1\sigma_1g_2\sigma_2), \pi_i(g_1\sigma_1)\pi_i(g_2\sigma_2)) \rightarrow 0$ as $i \rightarrow \mathcal{U}$, i.e., $\pi$ is a group homomorphism.

We now show $\pi$ is isometric to conclude the proof. By definition of $\pi_i$, we see that,
\begin{equation*}
    d(\pi_i(g_1\sigma_1), \pi_i(g_2\sigma_2)) = d(\sigma_1, \sigma_2) + \frac{1}{|A|} \sum_{\begin{smallmatrix}
    a \in A\\
    \sigma_1(a) = \sigma_2(a)
    \end{smallmatrix}} d(\phi_i(g_1[\sigma_1(a)]), \phi_i(g_2[\sigma_2(a)]))
\end{equation*}

For any fixed $a \in A$ with $\sigma_1(a) = \sigma_2(a)$, $g_1[\sigma_1(a)]$ is a fixed element of $G$ and so is $g_2[\sigma_2(a)]$, hence $d(\phi_i(g_1[\sigma_1(a)]), \phi_i(g_2[\sigma_2(a)])) = 0$ if $g_1[\sigma_1(a)] = g_2[\sigma_2(a)]$ and $d(\phi_i(g_1[\sigma_1(a)]), \phi_i(g_2[\sigma_2(a)])) \rightarrow 1$ as $i \rightarrow \mathcal{U}$ otherwise. Hence, as $i \rightarrow \mathcal{U}$, we have,
\begin{equation*}
\begin{split}
    d(\pi_i(g_1\sigma_1), \pi_i(g_2\sigma_2)) &\rightarrow d(\sigma_1, \sigma_2) + \frac{1}{|A|} |\{a \in A: \sigma_1(a) = \sigma_2(a)\textrm{ and }g_1[\sigma_1(a)] \neq g_2[\sigma_2(a)]\}\\
    &= d_{G, A}(g_1\sigma_1, g_2\sigma_2)
\end{split}
\end{equation*}

This proves the claim.
\end{proof}

As an immediate corollary, we get,

\begin{cor}\label{intermediate-ultraprod-cor}
Let $G$ be a countable discrete group. Suppose there exists a free ultrafilter $\mathcal{U}$ on $\mathbb{N}$, a sequence of sofic groups $G_i$, and a sequence of finite sets $A_i$ s.t. $G$ embeds into $\prod_\mathcal{U} (G_i^{\oplus A_i} \rtimes_{\beta_{A_i}} \Sym(A_i), d_{G_i, A_i})$, then $G$ is sofic.
\end{cor}

\begin{thm}\label{gen-wr-prod}
Let $G, H$ be sofic groups, $\alpha: H \curvearrowright X$ be a sofic action. Then the generalized wreath product $G \wr_{\alpha} H$ is sofic.
\end{thm}

\begin{proof} Fix increasing sequences of finite subsets $F_1 \subseteq F_2 \subseteq \cdots \subseteq H$ and $E_1 \subseteq E_2 \subseteq \cdots \subseteq X$ s.t. $\cup_i F_i = H$, $\cup_i E_i = X$. Fix a decreasing sequence $\epsilon_i > 0$ s.t. $\lim_i \epsilon_i = 0$. For each $i$, let $\varphi_i: H \rightarrow \Sym(A_i)$ be a unital, $(F_i, \epsilon_i)$-multiplicative, and an $(F_i, E_i, \epsilon_i)$-orbit approximation of $\alpha$. By definition of $(F_i, E_i, \epsilon_i)$-orbit approximation, there exists a finite set $B_i$ and a subset $S_i \subseteq A_i$ s.t. $|S_i| > (1 - \epsilon_i)|A_i|$ and for each $s \in S_i$ there is an injective map $\pi^i_s: E_i \hookrightarrow B_i$ s.t. $\pi^i_{\varphi_i(g)s}(x) = \pi^i_s(\alpha(g^{-1})x)$ for all $s \in S_i$, $g \in F_i$, $x \in E_i$, whenever $\varphi_i(g)s \in S_i$ and $\alpha(g^{-1})x \in E_i$. Let $G_i = G^{\oplus B_i}$, which is sofic as $G$ is. Let $p_i: G^{\oplus X} \rightarrow G^{\oplus E_i}$ be the canonical projection map. Let $q^i_s: G^{\oplus E_i} \hookrightarrow G^{\oplus B_i} = G_i$ be the inclusion map induced by $\pi^i_s: E_i \hookrightarrow B_i$. Let $P^i_s = q^i_s \circ p_i$.

Let the action of $H$ on $G^{\oplus X}$ via $\alpha$ be denoted by $\beta$, i.e., $G \wr_{\alpha} H = G^{\oplus X} \rtimes_\beta H$. We define $\rho_i: G \wr_{\alpha} H \rightarrow G_i^{\oplus A_i} \rtimes_{\beta_{A_i}} \Sym(A_i)$ as follows,
\begin{equation*}
    \rho_i(gh) = (\oplus_{s \in S_i} P^i_s(g) \bigoplus \oplus_{a \in A_i \setminus S_i} 1_{G_i}) \cdot \varphi_i(h)
\end{equation*}
where $g \in G^{\oplus X}$ and $h \in H$. Let $\mathcal{U}$ be an arbitrary free ultrafilter on $\mathbb{N}$. Let $\rho: G \wr_{\alpha} H \rightarrow \prod_\mathcal{U} (G_i^{\oplus A_i} \rtimes_{\beta_{A_i}} \Sym(A_i), d_{G_i, A_i})$ be given by $\rho(gh) = (\rho_i(gh))_\mathcal{U}$.

We shall now prove that $\rho$ is a group homomorphism. Indeed, let $g_1, g_2 \in G^{\oplus X}$, $h_1, h_2 \in H$. The supports of $g_1$ and $g_2$ are finite, so for large enough $i$, $\textrm{supp}(g_2) \subseteq E_i$, $\alpha(h_1) \cdot \textrm{supp}(g_2) \subseteq E_i$, and $h_1, h_1^{-1}, h_2 \in F_i$. Furthermore, whenever $h_1, h_1^{-1} \in F_i$, as $\varphi_i$ is $(F_i, \epsilon_i)$-multiplicative, we have,
\begin{equation*}
\begin{split}
    d(\varphi_i(h_1^{-1})^{-1}, \varphi_i(h_1)) &= d(\varphi_i(h_1^{-1})\varphi_i(h_1^{-1})^{-1}, \varphi_i(h_1^{-1})\varphi_i(h_1))\\
    &= d(1, \varphi_i(h_1^{-1})\varphi_i(h_1))\\
    &= d(\varphi_i(h_1h_1^{-1}), \varphi_i(h_1^{-1})\varphi_i(h_1))\\
    &< \epsilon_i
\end{split}
\end{equation*}

Let $\tilde{S}_i = \{s \in S_i: \varphi_i(h_1^{-1})^{-1}s = \varphi_i(h_1)s\}$. Then, for large enough $i$, $|\tilde{S}_i| > (1 - 2\epsilon_i)|A_i|$. Now, for large enough $i$,
\begin{equation*}
\begin{split}
    \rho_i(g_1h_1g_2h_2) &= \rho_i(g_1\beta_{h_1}(g_2)h_1h_2)\\
    &= (\oplus_{s \in S_i} P^i_s(g_1\beta_{h_1}(g_2)) \bigoplus \oplus_{a \in A_i \setminus S_i} 1_{G_i}) \cdot \varphi_i(h_1h_2)\\
    &= (\oplus_{s \in S_i} P^i_s(g_1)P^i_s(\beta_{h_1}(g_2)) \bigoplus \oplus_{a \in A_i \setminus S_i} 1_{G_i}) \cdot \varphi_i(h_1h_2)
\end{split}
\end{equation*}

We observe that, when $s \in \tilde{S}_i \cap \varphi_i(h_1)\tilde{S}_i = \tilde{S}_i \cap \varphi_i(h_1^{-1})^{-1}\tilde{S}_i$, $P^i_s(\beta_{h_1}(g_2)) \in G_i = G^{\oplus B_i}$, regarded as a function from $B_i$ to $G$, is given by,
\begin{equation*}
    P^i_s(\beta_{h_1}(g_2))(b) = \begin{cases}
        g_2(\alpha(h_1^{-1})x), \textrm{if }x \in E_i\textrm{ and }\pi^i_s(x) = b\\
        1_G, \textrm{otherwise}
    \end{cases}
\end{equation*}

We note that $\alpha(h_1^{-1})x \in \textrm{supp}(g_2)$ iff $x \in \alpha(h_1) \cdot \textrm{supp}(g_2)$. Since for large enough $i$, $\alpha(h_1) \cdot \textrm{supp}(g_2) \subseteq E_i$, $\textrm{supp}(g_2) \subseteq E_i$, $h_1^{-1} \in F_i$, and assuming $s \in \tilde{S}_i \cap \varphi_i(h_1)\tilde{S}_i$, we have,
\begin{equation*}
\begin{split}
    P^i_s(\beta_{h_1}(g_2))(b) &= \begin{cases}
        g_2(\alpha(h_1^{-1})x), \textrm{if }x \in \alpha(h_1) \cdot \textrm{supp}(g_2)\textrm{ and }\pi^i_s(x) = b\\
        1_G, \textrm{otherwise}
    \end{cases}\\
    &= \begin{cases}
        g_2(x), \textrm{if }x \in \textrm{supp}(g_2)\textrm{ and }\pi^i_s(\alpha(h_1)x) = b\\
        1_G, \textrm{otherwise}
    \end{cases}\\
    &= \begin{cases}
        g_2(x), \textrm{if }x \in \textrm{supp}(g_2)\textrm{ and }\pi^i_{\varphi_i(h_1^{-1})s}(x) = b\\
        1_G, \textrm{otherwise}
    \end{cases}\\
    &= \begin{cases}
        g_2(x), \textrm{if }x \in E_i\textrm{ and }\pi^i_{\varphi_i(h_1^{-1})s}(x) = b\\
        1_G, \textrm{otherwise}
    \end{cases}\\
    &= P^i_{\varphi_i(h_1^{-1})s}(g_2)(b)
\end{split}
\end{equation*}

So,
\begin{equation*}
    \rho_i(g_1h_1g_2h_2) = (\oplus_{s \in \tilde{S}_i \cap \varphi_i(h_1)\tilde{S}_i} P^i_s(g_1)P^i_{\varphi_i(h_1^{-1})s}(g_2) \bigoplus \oplus_{a \in A_i \setminus (\tilde{S}_i \cap \varphi_i(h_1)\tilde{S}_i)} \kappa_a) \cdot \varphi_i(h_1h_2)
\end{equation*}
for some $\kappa_a \in G_i$. On the other hand,
\begin{equation*}
\begin{split}
    \rho_i(g_1h_1)\rho_i(g_2h_2) &= (\oplus_{s \in S_i} P^i_s(g_1) \bigoplus \oplus_{a \in A_i \setminus S_i} 1_{G_i}) \cdot \varphi_i(h_1) \cdot (\oplus_{s \in S_i} P^i_s(g_2) \bigoplus \oplus_{a \in A_i \setminus S_i} 1_{G_i}) \cdot \varphi_i(h_2)\\
    &= (\oplus_{s \in \tilde{S}_i \cap \varphi_i(h_1)\tilde{S}_i} P^i_s(g_1)P^i_{\varphi_i(h_1)^{-1}s}(g_2) \bigoplus \oplus_{a \in A_i \setminus (\tilde{S}_i \cap \varphi_i(h_1)\tilde{S}_i)} \lambda_a) \cdot \varphi_i(h_1)\varphi_i(h_2)
\end{split}
\end{equation*}
for some $\lambda_a \in G_i$. Thus, for large enough $i$,
\begin{equation*}
\begin{split}
    d_{G_i, A_i}(\rho_i(g_1h_1g_2h_2), \rho_i(g_1h_1)\rho_i(g_2h_2)) &\leq d(\varphi_i(h_1h_2), \varphi_i(h_1)\varphi_i(h_2)) + \frac{A_i \setminus (\tilde{S}_i \cap \varphi_i(h_1)\tilde{S}_i)}{|A_i|}\\
    &< \epsilon_i + 4\epsilon_i\\
    &= 5\epsilon_i\\
    &\rightarrow 0
\end{split}
\end{equation*}

This proves that $\rho$ is a group homomorphism. Let $N = \textrm{ker}(\rho)$. By Corollary \ref{intermediate-ultraprod-cor}, $(G \wr_{\alpha} H)/N$ is sofic. Let $\iota: G \wr_{\alpha} H \rightarrow (G \wr_{\alpha} H)/N \times H$ be defined by $\iota(gh) = (ghN, h)$ where $g \in G^{\oplus X}$ and $h \in H$. Since both $(G \wr_{\alpha} H)/N$ and $H$ are sofic, it now suffices to prove $\iota$ is injective.

Clearly, $\textrm{ker}(\iota) \subseteq G^{\oplus X}$. Assume to the contrary that $\textrm{ker}(\iota)$ is not trivial and let $g \in \textrm{ker}(\iota) \setminus \{1\}$. Then $\rho_i(g) = (\oplus_{s \in S_i} P^i_s(g) \bigoplus \oplus_{a \in A_i \setminus S_i} 1_{G_i})$. Since $g \neq 1$, there exists $x \in X$ s.t. $g(x) \neq 1_G$. For large enough $i$, $x \in E_i$, and in such cases $P^i_s(g) \neq 1_{G_i}$ for all $s \in S_i$. Then we have,
\begin{equation*}
    d_{G_i, A_i}(\rho_i(g), 1) = \frac{|S_i|}{|A_i|} = 1 - \epsilon_i \rightarrow 1
\end{equation*}

In particular, $\rho(g) \neq 1$, so $g \notin N$. But then $\iota(g) \neq 1$, a contradiction. This proves the claim.
\end{proof}

The following theorem can be proved following similar arguments as in the proof of Theorem \ref{gen-wr-prod}. We include parts of a proof here for the convenience of the readers,

\begin{thm}\label{gen-free-wr-prod}
Let $H$ be a sofic group, $\alpha: H \curvearrowright X$ be a sofic action, $A \leq G$ be an inclusion of countable discrete groups s.t. the amalgamated free product of any countably many copies of $G$ over $A$ is sofic. Then the amalgamated free generalized wreath product $G \wr^{\ast_A}_\alpha H$ is sofic. In particular, if $G, H$ are sofic groups and $\alpha: H \curvearrowright X$ is a sofic action, then the free generalized wreath product $G \wr^\ast_\alpha H$ is sofic, and under the same conditions, if we in addition have an amenable subgroup $A$ of $G$, then the amalgamated free generalized wreath product $G \wr^{\ast_A}_\alpha H$ is sofic.
\end{thm}

\begin{proof}[Proof outline] Again, fix increasing sequences of finite subsets $F_1 \subseteq F_2 \subseteq \cdots \subseteq H$ and $E_1 \subseteq E_2 \subseteq \cdots \subseteq X$ s.t. $\cup_i F_i = H$, $\cup_i E_i = X$. Fix a decreasing sequence $\epsilon_i > 0$ s.t. $\lim_i \epsilon_i = 0$. For each $i$, let $\varphi_i: H \rightarrow \Sym(A_i)$ be a unital, $(F_i, \epsilon_i)$-multiplicative, and an $(F_i, E_i, \epsilon_i)$-orbit approximation of $\alpha$. By definition of $(F_i, E_i, \epsilon_i)$-orbit approximation, there exists a finite set $B_i$ and a subset $S_i \subseteq A_i$ s.t. $|S_i| > (1 - \epsilon_i)|A_i|$ and for each $s \in S_i$ there is an injective map $\pi^i_s: E_i \hookrightarrow B_i$ s.t. $\pi^i_{\varphi_i(g)s}(x) = \pi^i_s(\alpha(g^{-1})x)$ for all $s \in S_i$, $g \in F_i$, $x \in E_i$, whenever $\varphi_i(g)s \in S_i$ and $\alpha(g^{-1})x \in E_i$. Let $G_i = \ast^{b \in B_i}_A G_b$, where $G_b$ are copies of $G$. Under the assumption that the amalgamated free product of any countably many copies of $G$ over $A$ is sofic, we see that $G_i$ is sofic. Let $p_i: \ast^{x \in X}_A G_x \rightarrow \ast^{x \in E_i}_A G_x$, where $G_x$ are copies of $G$, be the map that is the identity map on $\ast^{x \in E_i}_A G_x$ and sends everything else to $1$. Let $q^i_s: \ast^{x \in E_i}_A G_x \hookrightarrow \ast^{b \in B_i}_A G_b = G_i$ be the inclusion map induced by $\pi^i_s: E_i \hookrightarrow B_i$. Let $P^i_s = q^i_s \circ p_i$.

Let the action of $H$ on $\ast^{x \in X}_A G_x$ via $\alpha$ be denoted by $\beta$, i.e., $G \wr^{\ast_A}_{\alpha} H = \ast^{x \in X}_A G_x \rtimes_\beta H$. We define $\rho_i: G \wr^{\ast_A}_{\alpha} H \rightarrow G_i^{\oplus A_i} \rtimes_{\beta_{A_i}} \Sym(A_i)$ as follows,
\begin{equation*}
    \rho_i(gh) = (\oplus_{s \in S_i} P^i_s(g) \bigoplus \oplus_{a \in A_i \setminus S_i} 1_{G_i}) \cdot \varphi_i(h)
\end{equation*}
where $g \in \ast^{x \in X}_A G_x$ and $h \in H$. Let $\mathcal{U}$ be an arbitrary free ultrafilter on $\mathbb{N}$. Let $\rho: G \wr_{\alpha} H \rightarrow \prod_\mathcal{U} (G_i^{\oplus A_i} \rtimes_{\beta_{A_i}} \Sym(A_i), d_{G_i, A_i})$ be given by $\rho(gh) = (\rho_i(gh))_\mathcal{U}$. The remainder of the proof follows the same outline as the proof of Theorem \ref{gen-wr-prod} - the only additional fact we note here is that, as the support of any element in $\ast^{x \in X}_A G_x$ is finite, it is eventually contained in $E_i$ for large enough $i$, and so $p_i$, and therefore $P^i_s = q^i_s \circ p_i$, is eventually multiplicative on any fixed finitely many elements of $\ast^{x \in X}_A G_x$.
\end{proof}

A natural setting where the first line of the above applies is arbitrary free products of free groups amalgamated over any fixed subgroup (see \cite{gao2021relative, mjgaosri}).

For the following results in the Hilbert-Schmidt setting, they can be proved using the same line of reasoning, where the map $p_i$ is replaced by the conditional expectation from $M^{\bar{\otimes} X}$ (or $\ast^{x \in X}_N M_x$) to $M^{\bar{\otimes} E_i}$ (or $\ast^{x \in E_i}_N M_x$, resp.) $G_i$ shall now be replaced by $M_i = M^{\bar{\otimes} B_i}$ (or $M_i = \ast^{b \in B_i}_N M_b$, resp.) $q^i_s$ shall still be the natural inclusion map induced by $\pi^i_s: E_i \hookrightarrow B_i$ and $P^i_s = q^i_s \circ p_i$. The map $\rho_i$ shall now be a map from $M \wr_\alpha H$ (or $M \wr^{\ast_N}_\alpha H$, resp.) to $M_i^{\bar{\otimes} A_i} \bar{\otimes} \mathbb{M}_{|A_i|}(\mathbb{C}) \bar{\otimes} L(H)$ given by,

\begin{equation*}
    \rho_i(mh) = [(\oplus_{s \in S_i} P^i_s(m) \bigoplus \oplus_{a \in A_i \setminus S_i} 0) \cdot \varphi_i(h)] \otimes \lambda_h
\end{equation*}
where $m \in M^{\bar{\otimes} X}$ (or $\ast^{x \in X}_N M_x$, resp.) and $h \in H$, and where $\oplus_{s \in S_i} P^i_s(m) \bigoplus \oplus_{a \in A_i \setminus S_i} 0$ is understood as a diagonal matrix in $M_i^{\bar{\otimes} A_i} \bar{\otimes} \mathbb{M}_{|A_i|}(\mathbb{C})$, $\varphi_i(h)$ is now interpreted as a permutation matrix in $\mathbb{M}_{|A_i|}(\mathbb{C})$, and $\lambda_h$ is the unitary associated with $h$ in $L(H)$. The remainder of the proof follows essentially the same outline, replacing the arguments verifying the conditions on the $d_{G_i, A_i}$ metric with the conditions on the preservation of the trace. (The final part of the argument, dealing with the kernel of $\rho$, is no longer needed, as we included the $\otimes \lambda_h$ term in the definition of $\rho_i$.)

\begin{thm}\label{hyper1}
Let $(M, \tau)$ be a Connes-embeddable tracial von Neumann algebra, $H$ be a hyperlinear group, $\alpha: H \curvearrowright X$ be a sofic action. Then the generalized wreath product $M \wr_{\alpha} H$ is Connes-embeddable. In particular, if $G, H$ are hyperlinear groups and $\alpha: H \curvearrowright X$ is a sofic action, then the generalized wreath product $G \wr_{\alpha} H$ is hyperlinear.
\end{thm}

\begin{thm}\label{hyper2}
Let $H$ be a hyperlinear group, $\alpha: H \curvearrowright X$ be a sofic action, $N \subseteq M$ be an inclusion of tracial von Neumann algebras s.t. the amalgamated free product of any countably many copies of $M$ over $N$ is Connes-embeddable. Then the amalgamated free generalized wreath product $M \wr^{\ast_N}_\alpha H$ is Connes-embeddable. In particular, if $H$ is a hyperlinear group, $\alpha: H \curvearrowright X$ is a sofic action, and $M$ is a Connes-embeddable tracial von Neumann algebra, then the free generalized wreath product $M \wr^\ast_\alpha H$ is Connes-embeddable, and under the same conditions, if we in addition have an amenable subalgebra $N$ of $M$, then the amalgamated free generalized wreath product $M \wr^{\ast_N}_\alpha H$ is Connes-embeddable.
\end{thm}

\begin{cor}\label{hyper3}
Let $H$ be a hyperlinear group, $\alpha: H \curvearrowright X$ be a sofic action, $A \leq G$ be an inclusion of countable discrete groups s.t. the amalgamated free product of any countably many copies of $G$ over $A$ is hyperlinear. Then the amalgamated free generalized wreath product $G \wr^{\ast_A}_\alpha H$ is hyperlinear. In particular, if $G, H$ are hyperlinear groups and $\alpha: H \curvearrowright X$ is a sofic action, then the free generalized wreath product $G \wr^\ast_\alpha H$ is hyperlinear, and under the same conditions, if we in addition have an amenable subgroup $A$ of $G$, then the amalgamated free generalized wreath product $G \wr^{\ast_A}_\alpha H$ is hyperlinear.
\end{cor}

\section{Concluding remarks and open questions}

We  document the following Proposition which places an aspect of our work in the context of Elek-Lippner's work \cite{eleklippner} defining soficity for equivalence relations and of Paunescu \cite{Paunescusofic1}. We omit the proof because it is substantially similar to the arguments in the previous section. It is open whether the converse of the statement below  holds.

\begin{prop}
Let $(\Omega, \mu)$ be a standard probability space, $G$ be a sofic group, $\alpha: G \curvearrowright X$ be a sofic action. Then the induced generalized Bernoulli shift $G \curvearrowright (\Omega^X, \mu^{\otimes |X|})$ is sofic in the sense of \cite{Paunescusofic1}.
\end{prop}

We ask the following two natural questions on the permanence of sofic actions which we are currently unable to answer:

\begin{question}\label{direct sum}
Suppose we have actions $\alpha_i: G_i \curvearrowright X$ which commute with each other and where $i$ ranges over a countable index set. Then the actions naturally give rise to an action $\alpha: \oplus_i G_i \curvearrowright X$. $\alpha$ is sofic iff all $\alpha_i$ are sofic?
\end{question}

\begin{question}
Suppose we have actions $\alpha_i: G_i \curvearrowright X$ where $i$ ranges over a countable index set. Then the actions naturally give rise to an action $\alpha: \ast_i G_i \curvearrowright X$. $\alpha$ is sofic iff all $\alpha_i$ are sofic?
\end{question}

The forward directions of both conjectures follow from item 3 of Proposition \ref{basic-prop}. By item 4 of Proposition \ref{basic-prop}, it suffices to consider the case where there are only two groups $G_1$ and $G_2$. We document below one more natural question which is nothing but the converse of Theorem \ref{gen-wr-prod}: 

\begin{question}Let $G, H$ be nontrivial countable groups, $\alpha: H \curvearrowright X$ be an action. Then the generalized wreath product $G \wr_{\alpha} H$ is sofic iff $G$ and $H$ are sofic and the action $\alpha$ is sofic?
\end{question}

A positive answer to Question \ref{direct sum} implies a positive answer to the above. Indeed, let $\Gamma$ be a sofic group. Then the left multiplication $\alpha_1: \Gamma \curvearrowright \Gamma$ is sofic by Theorem \ref{thm-left-multi}. The right multiplication action $\alpha_2: \Gamma \curvearrowright \Gamma$, $\alpha_2(\gamma)\eta = \eta\gamma^{-1}$ is isomorphic to the left multiplication action, so as such is also sofic. $\alpha_1$ and $\alpha_2$ clearly commute, so the combined action $\alpha: \Gamma \oplus \Gamma \curvearrowright \Gamma$ is sofic assuming a positive answer to Question \ref{direct sum}. Let $\Delta: \Gamma \rightarrow \Gamma \oplus \Gamma$ be the diagonal embedding, then by item 3 of Proposition \ref{basic-prop}, $\alpha \circ \Delta$ is sofic. One can easily verify that $\alpha \circ \Delta$ is the conjugation action of $\Gamma$ on itself. That is, the conjugation action of any sofic group on itself is sofic, assuming Question \ref{direct sum} has a positive answer.

Now, if $G \wr_{\alpha} H$ is sofic, then the above applies to it, so its conjugation action on itself is sofic. By item 3 of Proposition \ref{basic-prop}, we have the conjugation action $\beta: H \curvearrowright G \wr_{\alpha} H$ is sofic. Now, as $G$ is nontrivial, we may fix $g \in G \setminus \{1_G\}$. For any $x \in X$, let $g_x \in G^{\oplus X} \subseteq G \wr_{\alpha} H$ be the element that, as a function from $X$ to $G$, takes value $1_G$ at all elements of $X$ except $x$ and takes value $g$ at $x$. Let the orbit of $g_x$ under $\beta$ be $\mathcal{O}(g_x)$ and the orbit of $x$ under $\alpha$ be $\mathcal{O}(x)$. Since $\beta(h)g_x = hg_xh^{-1} = g_{\alpha(h)x}$ and furthermore that $g_x \neq g_y$ whenever $x \neq y$ as $g \neq 1_G$, we easily see that the map $\mathcal{O}(x) \ni y \mapsto g_y \in \mathcal{O}(g_x)$ is a bijection that identifies $\beta$ restricted to $\mathcal{O}(g_x)$ and $\alpha$ restricted to $\mathcal{O}(x)$. As $\beta$ is sofic, this implies, via item 2 of Proposition \ref{basic-prop}, that $\alpha$ restricted to $\mathcal{O}(x)$ is sofic. Since $x \in X$ is arbitrary, $\alpha$ restricted to each of its orbit is sofic, whence $\alpha$ is sofic by Proposition \ref{transitive-prop}.

\bibliographystyle{amsalpha}
\bibliography{inneramen}

\providecommand{\bysame}{\leavevmode\hbox to3em{\hrulefill}\thinspace}
\providecommand{\MR}{\relax\ifhmode\unskip\space\fi MR }
\providecommand{\MRhref}[2]{%
  \href{http://www.ams.org/mathscinet-getitem?mr=#1}{#2}
}
\providecommand{\href}[2]{#2}
\begin{thebibliography}{DKP14}

\bibitem[Bow10]{Bowenfinvariant}
Lewis~Phylip Bowen, \emph{A measure-conjugacy invariant for free group actions}, Ann. of Math. (2) \textbf{171} (2010), no.~2, 1387--1400. \MR{2630067}

\bibitem[Con76]{Connes}
A.~Connes, \emph{Classification of injective factors. {C}ases {$II_{1},$} {$II_{\infty },$} {$III_{\lambda },$} {$\lambda \not=1$}}, Ann. of Math. (2) \textbf{104} (1976), no.~1, 73--115. \MR{454659}

\bibitem[DKP14]{DKPsofic}
Ken Dykema, David Kerr, and Mika\"{e}l Pichot, \emph{Sofic dimension for discrete measured groupoids}, Trans. Amer. Math. Soc. \textbf{366} (2014), no.~2, 707--748. \MR{3130315}

\bibitem[EL10a]{ElekLip}
G.~Elek and G.~Lippner, \emph{Sofic equivalence relations}, J. Funct. Anal. \textbf{258} (2010), 1692--1708.

\bibitem[EL10b]{eleklippner}
G\'{a}bor Elek and G\'{a}bor Lippner, \emph{Sofic equivalence relations}, J. Funct. Anal. \textbf{258} (2010), no.~5, 1692--1708. \MR{2566316}

\bibitem[ES04]{elekszabodirect}
G\'{a}bor Elek and Endre Szab\'{o}, \emph{Sofic groups and direct finiteness}, J. Algebra \textbf{280} (2004), no.~2, 426--434. \MR{2089244}

\bibitem[ES11]{elekszabosofic}
\bysame, \emph{Sofic representations of amenable groups}, Proc. Amer. Math. Soc. \textbf{139} (2011), no.~12, 4285--4291. \MR{2823074}

\bibitem[GEMss]{mjgaosri}
David Gao, Srivatsav~Kunnawalkam Elayavalli, and Mahan Mj, \emph{Remarks on soficity of amalgamated free products}, in progress.

\bibitem[GJ21]{gao2021relative}
David Gao and Marius Junge, \emph{Relative embeddability of von neumann algebras and amalgamated free products}, 2021.

\bibitem[HE23]{hayes2023sofic}
Ben Hayes and Srivatsav~Kunnawalkam Elayavalli, \emph{On sofic approximations of non amenable groups}, 2023.

\bibitem[HS18]{HayesSale}
Ben Hayes and Andrew~W. Sale, \emph{Metric approximations of wreath products}, Ann. Inst. Fourier (Grenoble) \textbf{68} (2018), no.~1, 423--455. \MR{3795485}

\bibitem[KL11]{KLi}
David Kerr and Hanfeng Li, \emph{Entropy and the variational principle for actions of sofic groups}, Invent. Math. \textbf{186} (2011), no.~3, 501--558. \MR{2854085}

\bibitem[L{\"{u}}c02]{Luck}
W.~L{\"{u}}ck, \emph{$l^{2}$-invariants: Theory and applications to geometry and $k$-theory}, Springer-Verlag, Berlin, 2002.

\bibitem[Pau14]{PaunConvex}
Liviu Paunescu, \emph{A convex structure on sofic embeddings}, Ergodic Theory Dynam. System \textbf{34} (2014), no.~4, 1343--1352.

\bibitem[Pes08]{surveysofic}
Vladimir~G. Pestov, \emph{Hyperlinear and sofic groups: a brief guide}, Bull. Symbolic Logic \textbf{14} (2008), no.~4, 449--480. \MR{2460675}

\bibitem[Pop14]{Popaindep}
Sorin Popa, \emph{Independence properties in subalgebras of ultraproduct {$\rm II_1$} factors}, J. Funct. Anal. \textbf{266} (2014), no.~9, 5818--5846. \MR{3182961}

\bibitem[P\u11]{Paunescusofic1}
Liviu P\u{a}unescu, \emph{On sofic actions and equivalence relations}, J. Funct. Anal. \textbf{261} (2011), no.~9, 2461--2485. \MR{2826401}

\bibitem[Tho08]{ThomDiophantine}
Andreas Thom, \emph{Sofic groups and {D}iophantine approximation}, Comm. Pure Appl. Math. \textbf{61} (2008), no.~8, 1155--1171. \MR{2417890}

\end{thebibliography}

\end{document}